\documentclass[10pt]{article}

\usepackage{amsmath}
\usepackage{amssymb}
\usepackage{amsthm}
\usepackage{graphics}
\usepackage[latin1]{inputenc}
\usepackage[english]{babel}
\usepackage[T1]{fontenc}

\usepackage{enumerate}

\newtheorem{thm}{Theorem}[section]
\newtheorem{prop}{Proposition}[section]

\newtheorem{lemma}{Lemma}[section]
\newtheorem{rem}{Remark}[section]

\newtheorem{defi}{Definition}[section]

\newcommand{\R}{\mathbb{R}}             
\newcommand{\C}{\mathbb{C}}             

\newcommand{\D}{\mathcal{D}}            

\newcommand{\Section}[1]{\section{#1} \setcounter{equation}{0}}

\begin{document}

\title{On the hidden mechanism behind non-uniqueness for the anisotropic Calder\'on problem with data on disjoint sets}
\author{Thierry Daud\'e \footnote{Research supported by the French National Research Projects AARG, No. ANR-12-BS01-012-01, and Iproblems, No. ANR-13-JS01-0006} $^{\,1}$, Niky Kamran \footnote{Research supported by NSERC grant RGPIN 105490-2011} $^{\,2}$ and Francois Nicoleau \footnote{Research supported by the French National Research Project NOSEVOL, No. ANR- 2011 BS0101901} $^{\,3}$\\[12pt]
 $^1$  \small D\'epartement de Math\'ematiques. UMR CNRS 8088, Universit\'e de Cergy-Pontoise, \\
 \small 95302 Cergy-Pontoise, France. \\
\small Email: thierry.daude@u-cergy.fr \\
$^2$ \small Department of Mathematics and Statistics, McGill University,\\ \small  Montreal, QC, H3A 2K6, Canada. \\
\small Email: nkamran@math.mcgill.ca \\
$^3$  \small  Laboratoire de Math\'ematiques Jean Leray, UMR CNRS 6629, \\ \small 2 Rue de la Houssini\`ere BP 92208, F-44322 Nantes Cedex 03. \\
\small Email: francois.nicoleau@math.univ-nantes.fr }





\maketitle


\begin{abstract}

We show that there is generically non-uniqueness for the anisotropic Calder\'on problem at fixed frequency when the Dirichlet and Neumann data are measured on disjoint sets of the boundary of a given domain. More precisely, we first show that given a smooth compact connected Riemannian manifold with boundary $(M,g)$ of dimension $n\geq 3$, there exist in the conformal class of $g$ an infinite number of Riemannian metrics $\tilde{g}$ such that their corresponding DN maps at a fixed frequency coincide when the Dirichlet data $\Gamma_D$ and Neumann data $\Gamma_N$ are measured on disjoint sets and satisfy $\overline{\Gamma_D \cup \Gamma_N} \ne \partial M$. The conformal factors that lead to these non-uniqueness results for the anisotropic Calder\'on problem satisfy a nonlinear elliptic PDE of Yamabe type on the original manifold $(M,g)$ and are associated to a natural but subtle gauge invariance of the anisotropic Calder\'on problem with data on disjoint sets. We then construct a large class of counterexamples to uniqueness in dimension $n\geq 3$ to the anisotropic Calder\'on problem at fixed frequency with data on disjoint sets and \emph{modulo this gauge invariance}. This class consists in cylindrical Riemannian manifolds with boundary having two ends (meaning that the boundary has two connected components), equipped with a suitably chosen warped product metric.


\vspace{0.5cm}




\end{abstract}

\tableofcontents


\Section{Introduction}

\subsection{The anisotropic Calder\'on problem}

The anisotropic Calder\'on problem on smooth compact connected Riemannian manifolds with boundary is a model example of an inverse problem which consists in recovering the physical properties of a medium (like its electrical conductivity) by making only electrical measurements at its boundary. In this paper, we consider the case where the Dirichlet and Neumann data are measured on \emph{disjoint} subsets of the boundary, an inverse problem which is important from a practical point of view and which is still largely open \cite{GT2, IUY2, KS1, KS2, KLO, LO1, LO2}. In order to state our results, we first recall the geometric formulation of the Calder\'on problem due Lee and Uhlmann \cite{LeU}. We refer to the surveys \cite{GT2, KS2, Sa, U1} for the current state of the art on the anisotropic Calder\'on problem and also to \cite{DSFKSU, DSFKLS, GSB, GT1, KS1, LaTU, LaU, LeU} for important contributions to the subject.

Let $(M, g)$ be an $n$ dimensional smooth compact connected Riemannian manifold with smooth boundary $\partial M$. Let us denote by $\Delta_{LB}$ the positive Laplace-Beltrami operator on $(M,g)$. In a local coordinate system $(x^i)_{i = 1,\dots,n}$, the Laplace-Beltrami operator $\Delta_{LB}$ is given by
$$
\Delta_{LB}=  -\Delta_g = -\frac{1}{\sqrt{|g|}} \partial_i \left( \sqrt{|g|} g^{ij} \partial_j \right),
$$
where  $|g| = \det \left(g_{ij}\right)$ is the determinant of the metric tensor $(g_{ij})$, where $\left(g^{ij}\right)$ is the inverse of $(g_{ij})$ and where we use the Einstein summation convention. We recall that the Laplace-Beltrami operator $-\Delta_g$ with Dirichlet boundary conditions is selfadjoint on $L^2(M, dVol_g)$ and has pure point spectrum $\{ \lambda_j\}_{j \geq 1}$ with $0 < \lambda_1 < \lambda_2 \leq \dots \leq \lambda_j \to +\infty$ (see for instance \cite{KKL}).

We consider the Dirichlet problem at a frequency $\lambda \in \R$ on $(M,g)$ such that $\lambda \notin \{ \lambda_j\}_{j \geq 1}$. We are interested thus in the solutions $u$ of
\begin{equation} \label{Eq00}
  \left\{ \begin{array}{cc} -\Delta_g u = \lambda u, & \textrm{on} \ M, \\ u = \psi, & \textrm{on} \ \partial M. \end{array} \right.
\end{equation}
It is well known (see for instance \cite{Sa, Ta1}) that for any $\psi \in H^{1/2}(\partial M)$, there exists a unique weak solution $u \in H^1(M)$ of (\ref{Eq00}). This allows us to define the Dirichlet-to-Neumann (DN) map as the operator $\Lambda_{g}(\lambda)$ from $H^{1/2}(\partial M)$ to $H^{-1/2}(\partial M)$ defined for all $\psi \in H^{1/2}(\partial M)$ by
\begin{equation} \label{DN-Abstract}
  \Lambda_{g}(\lambda) (\psi) = \left( \partial_\nu u \right)_{|\partial M},
\end{equation}
where $u$ is the unique solution of (\ref{Eq00}) and $\left( \partial_\nu u \right)_{|\partial M}$ is its normal derivative with respect to the unit outer normal vector $\nu$ on $\partial M$. Here $\left( \partial_\nu u \right)_{|\partial M}$ is interpreted in the weak sense as an element of $H^{-1/2}(\partial M)$ by
$$
  \left\langle \Lambda_{g}(\lambda) \psi | \phi \right \rangle = \int_M \langle du, dv \rangle_g \, dVol_g,
$$
for any $\psi \in H^{1/2}(\partial M)$ and $\phi \in H^{1/2}(\partial M)$ such that $u$ is the unique solution of (\ref{Eq00}) and $v$ is any element of $H^1(M)$ such that $v_{|\partial M} = \phi$. If $\psi$ is sufficiently smooth, we can check that
$$
  \Lambda_{g}(\lambda) \psi = g(\nu, \nabla u)_{|\partial M} = du(\nu)_{|\partial M} = \nu(u)_{|\partial M},
$$
where $\nu$ represents the unit outer normal vector to $\partial M$, so that an expression in local coordinates for the normal derivative is thus given by
\begin{equation} \label{DN-Coord}
\partial_\nu u = \nu^i \partial_i u.
\end{equation}

We shall be interested in the \emph{partial} DN maps defined as follows. Let $\Gamma_D$ and $\Gamma_N$ be two open subsets of $\partial M$. We define the partial DN map $\Lambda_{g,\Gamma_D,\Gamma_N}(\lambda)$ as the restriction of the global DN map $\Lambda_g(\lambda)$ to Dirichlet data given on $\Gamma_D$ and Neumann data measured on $\Gamma_N$. Precisely, consider the Dirichlet problem
\begin{equation} \label{Eq0}
  \left\{ \begin{array}{cc} -\Delta_g u = \lambda u, & \textrm{on} \ M, \\ u = \psi, & \textrm{on} \ \Gamma_D, \\ u = 0, & \textrm{on} \ \partial M \setminus \Gamma_D. \end{array} \right.
\end{equation}
We define $\Lambda_{g,\Gamma_D,\Gamma_N}(\lambda)$ as the operator acting on the functions $\psi \in H^{1/2}(\partial M)$ with $\textrm{supp}\,\psi \subset \Gamma_D$ by
\begin{equation} \label{Partial-DNmap}
  \Lambda_{g,\Gamma_D,\Gamma_N}(\lambda) (\psi) = \left( \partial_\nu u \right)_{|\Gamma_N},
\end{equation}
where $u$ is the unique solution of (\ref{Eq0}).

In its simplest form, the anisotropic partial Calder\'on problem can be stated as follows: \emph{Does the knowledge of the partial DN map $\Lambda_{g,\Gamma_D, \Gamma_N}(\lambda)$ at a fixed frequency $\lambda$ determine uniquely the metric $g$}?

The answer to the above question is negative because of a number of natural gauge invariances that are inherent to the problem. Indeed, it follows from the definition (\ref{Eq0}) - (\ref{Partial-DNmap}) that in any dimension, the partial DN map $\Lambda_{g, \Gamma_D, \Gamma_N}(\lambda)$ is invariant under pullback of the metric by the diffeomorphisms of $M$ that restrict to the identity on $\Gamma_D \cup \Gamma_N$, \textit{i.e.}
\begin{equation} \label{Inv-Diff}
  \forall \phi \in \textrm{Diff}(M) \ \textrm{such that} \ \phi_{|\Gamma_D \cup \Gamma_N} = Id, \quad \Lambda_{\phi^*g, \Gamma_D, \Gamma_N}(\lambda) = \Lambda_{g, \Gamma_D, \Gamma_N}(\lambda).
\end{equation}

In the two dimensional case and for zero frequency $\lambda = 0$, there is an additional gauge invariance of the DN map due to the fact that the Laplace-Beltrami operator is acted on by scalings under conformal changes of the metric.  More precisely, recall that if $\dim M=2$, then
$$
\Delta_{cg} = \frac{1}{c} \Delta_g,
$$
for any smooth function $c >0$. Therefore, we have in dimension $2$
\begin{equation} \label{Inv-Conf}
\forall c \in C^\infty(M) \ \textrm{such that} \ c >0 \ \textrm{and} \ c_{|\Gamma_N} = 1, \quad \Lambda_{c g, \Gamma_D, \Gamma_N}(0) = \Lambda_{g, \Gamma_D, \Gamma_N}(0),
\end{equation}
since the unit outer normal vectors $\nu_{cg}$ and $\nu_g$ coincide on $\Gamma_N$ in that case.

It therefore follows that the appropriate question to address (called the \emph{anisotropic Calder\'on conjecture}) is the following. \\

\noindent \textbf{(Q1)}: \emph{Let $M$ be a smooth compact connected manifold with smooth boundary $\partial M$ and let $g,\, \tilde{g}$ be smooth Riemannian metrics on $M$. Let $\Gamma_D, \Gamma_N$ be any open subsets of $\partial M$ and assume that $\lambda \in \R$ does not belong to $\sigma(-\Delta_g) \cup \sigma(-\Delta_{\tilde{g}})$. If
$$
  \Lambda_{g,\Gamma_D, \Gamma_N}(\lambda) = \Lambda_{\tilde{g},\Gamma_D, \Gamma_N}(\lambda),
$$
is it true that
$$
  g = \tilde{g},
$$
up to the gauge invariance (\ref{Inv-Diff}) if $\dim M \geq 3$ and up to the gauge invariances (\ref{Inv-Diff}) - (\ref{Inv-Conf}) if $\dim M = 2$ and $\lambda = 0$}? \\

There are three subcases of the above problem which are of particular interest:
\begin{itemize}
\item \textbf{Full data}: $\Gamma_D = \Gamma_N = \partial M$. In that case, we denote the DN map simply by $\Lambda_g(\lambda)$.

\item \textbf{Local data}: $\Gamma_D = \Gamma_N = \Gamma$, where $\Gamma$ can be any nonempty open subset of $\partial M$. In that case, we denote the DN map by $\Lambda_{g, \Gamma}(\lambda)$.

\item \textbf{Data on disjoint sets}: $\Gamma_D$ and $\Gamma_N$ are disjoint open sets of $\partial M$.
\end{itemize}

If $\dim M \geq 3$, one may also consider a simpler inverse problem by assuming that the Riemannian manifolds $(M,g)$ and $(M,\tilde{g})$ belong to the same conformal class, that is $\tilde{g} = c g$ for some smooth strictly positive function c. In that case, $g$ is considered as a given known background metric and the problem consists in determining the unknown scalar function $c$ from the DN map $\Lambda_{c g,\Gamma_D, \Gamma_N}(\lambda)$. In that case, the anisotropic Calder\'on problem becomes: \\

\noindent \textbf{(Q2)}: \emph{Let $(M,g)$ be a smooth compact connected Riemannian manifold of dimension $n\geq 3$ with smooth boundary $\partial M$ and let $\Gamma_D, \Gamma_N$ be open subsets of $\partial M$. Let $c$ be a smooth strictly positive function on $M$ and assume that $\lambda \in \R$ does not belong to $\sigma(-\Delta_g) \cup \sigma(-\Delta_{c g})$. If
$$
  \Lambda_{c g,\Gamma_D, \Gamma_N}(\lambda) = \Lambda_{g,\Gamma_D, \Gamma_N}(\lambda),
$$
does there exist a diffeomorphism $\phi: \, M \longrightarrow M$ with $\phi_{| \, \Gamma_D \cup \Gamma_N} = Id$ such that}
\begin{equation} \label{Inv-Conformal}
  \phi^* g = c g?
\end{equation}
Since any diffeomorphism $\phi: \, M \longrightarrow M$  which satisfies $\phi^* g = c g$ and $\phi_{|\Gamma} = Id$ for a non-empty open subset $\Gamma$ of $\partial M$ must be the identity \cite{Li}\footnote{Although Proposition 3.3 in \cite{Li} has been stated in the case $\Gamma = \partial M$, the result remains true when $\Gamma$ is replaced by any non-empty open subset of $\partial M$ }, we see that there is no ambiguity arising from diffeomorphisms in the solution of the anisotropic Calder\'on problem \textbf{(Q2)}. The condition (\ref{Inv-Conformal}) may therefore be replaced by the condition
\begin{equation} \label{Inv-Conformal-1}
  c = 1, \quad \textrm{on} \ M.
\end{equation}

A third version of the anisotropic Calder\'on problem which is somewhat related to \textbf{(Q2)}, but involves now an external potential, is given by the following. Consider the solution of the Schr\"odinger equation on $(M,g)$ with potential $V \in L^\infty(M)$
\begin{equation} \label{Eq0-Schrodinger}
  \left\{ \begin{array}{cc} (-\Delta_g + V) u = \lambda u, & \textrm{on} \ M, \\ u = \psi, & \textrm{on} \ \Gamma_D, \\ u = 0, & \textrm{on} \ \partial M \setminus \Gamma_D. \end{array} \right.
\end{equation}
It is well known (see for example \cite{DSFKSU, Sa}) that if $\lambda$ does not belong to the Dirichlet spectrum of $-\Delta_g +V$, then for any $\psi \in H^{1/2}(\partial M)$, there exists a unique weak solution $u \in H^1(M)$ of (\ref{Eq0-Schrodinger}). This allows us to define the partial Dirichlet-to-Neumann map $\Lambda_{g, V, \,\Gamma_D, \Gamma_N}(\lambda)$ for all $\psi \in H^{1/2}(\partial M)$ with supp $\psi \subset \Gamma_D$ by
\begin{equation} \label{DN-Abstract-Schrodinger}
  \Lambda_{g, V,\Gamma_D, \Gamma_N}(\lambda) (\psi) = \left( \partial_\nu u \right)_{|\Gamma_N},
\end{equation}
where $u$ is the unique solution of (\ref{Eq0-Schrodinger}) and $\left( \partial_\nu u \right)_{|\Gamma_N}$ is its normal derivative with respect to the unit outer normal vector $\nu$ on $\Gamma_N$. We assume again here that $g$ is a given background metric and the problem consists in determining the unknown potential $V \in L^\infty(M)$ from the DN map $\Lambda_{g, V, \,\Gamma_D, \Gamma_N}(\lambda)$. Precisely, the question is: \\

\noindent \textbf{(Q3)}: \emph{Let $(M,g)$ be a smooth compact connected Riemannian manifold with smooth boundary $\partial M$ and let $\Gamma_D, \Gamma_N$ be open subsets of $\partial M$. Let $V_1$ and $V_2$ be potentials in $L^\infty(M)$ and assume that  $\lambda \in \R$ does not belong to the Dirichlet spectra of $-\triangle_g + V_1$ and $-\triangle_g + V_2$. If
$$
  \Lambda_{g, V_1, \Gamma_D, \Gamma_N}(\lambda) = \Lambda_{g, V_2, \Gamma_D, \Gamma_N}(\lambda),
$$
is it true that}
$$
  V_1 = V_2?
$$

If $\dim M \geq 3$, there is a straightforward link between \textbf{(Q2)} and \textbf{(Q3)} that is based on the transformation law for the Laplace-Beltrami operator under conformal changes of metric,
\begin{equation} \label{ConformalScaling}
  -\Delta_{c^4 g} u = c^{-(n+2)} \left( -\Delta_g + q_{g,c} \right) \left( c^{n-2} u \right),
\end{equation}
where
\begin{equation} \label{q}
  q_{g,c} = c^{-n+2} \Delta_{g} c^{n-2}.
\end{equation}	
We have:

\begin{prop} \label{Link-c-to-V}
Let $\lambda \in \R$ be fixed. Assume that $c$ is a smooth strictly positive function on $M$ such that $c = 1$ on $\Gamma_D \cup \Gamma_N$. \\
1. If $\Gamma_D \cap \Gamma_N = \emptyset$, then
\begin{equation} \label{Link}
	\Lambda_{c^4 g, \Gamma_D, \Gamma_N}(\lambda) = \Lambda_{g, V_{g,c,\lambda}, \Gamma_D,\Gamma_N}(\lambda),
\end{equation}
where
\begin{equation} \label{Vgc}
  V_{g,c,\lambda} = q_{g,c} + \lambda(1-c^4), \quad q_{g,c} = c^{-n+2} \Delta_{g} c^{n-2}.
\end{equation}	
2. If $\Gamma_D \cap \Gamma_N \ne \emptyset$ and $\partial_{\nu} c = 0$ on $\Gamma_N$, then (\ref{Link}) also holds.
\end{prop}

\begin{proof}

Given a function $c$ satisfying the assumptions of the Proposition, consider the Dirichlet problem at fixed frequency $\lambda$ associated to the metric $c^4 g$, \textit{i.e.}
\begin{equation} \label{z1}
  \left\{ \begin{array}{cc} -\Delta_{c^4 g} u = \lambda u, & \textrm{on} \ M, \\ u = \psi, & \textrm{on} \ \Gamma_D, \\ u = 0, & \textrm{on} \ \partial M \setminus \Gamma_D. \end{array} \right.
\end{equation}
Using (\ref{ConformalScaling}) and setting $v = c^{n-2} u$, the Dirichlet problem (\ref{z1}) is equivalent to
\begin{equation} \label{z3}
  \left\{ \begin{array}{cc} (-\Delta_{g} + q_{g,c} + \lambda (1-c^4)) v = \lambda v, & \textrm{on} \ M, \\ v = c^{n-2} \psi, & \textrm{on} \ \Gamma_D, \\ v = 0, & \textrm{on} \ \partial M \setminus \Gamma_D. \end{array} \right.
\end{equation}
Since $c = 1$ on $\Gamma_D$, we see that the function $v$ satisfies
\begin{equation} \label{z4}
  \left\{ \begin{array}{cc} (-\Delta_{g} + V_{g,c,\lambda}) v = \lambda v, & \textrm{on} \ M, \\ v = \psi, & \textrm{on} \ \Gamma_D, \\ v = 0, & \textrm{on} \ \partial M \setminus \Gamma_D. \end{array} \right.
\end{equation}
where $V_{g,c,\lambda}$ is given by (\ref{Vgc}). In other words, $v$ is the unique solution of the Dirichlet problem (\ref{z4}) at frequency $\lambda$ associated to the Schr\"odinger operator $-\triangle_g + V_{g,c,\lambda}$.

Let us show now that $\Lambda_{c^4 g,\Gamma_D, \Gamma_N}(\lambda) = \Lambda_{g, V_{g,c,\lambda}, \Gamma_D, \Gamma_N}(\lambda)$ in the different cases stated in the Proposition.

On one hand, since the conformal factor $c$ satisfies $c = 1$ on $\Gamma_N$, the unit outgoing normal vector $\tilde{\nu}$ associated to $\tilde{g} = c^4 g$  is equal to the unit  outgoing normal vector $\nu$ associated to $g$ on $\Gamma_N$. Thus by definition of the partial DN map, we have
\begin{equation} \label{z5}
  \Lambda_{c^4 g,\Gamma_D,\Gamma_N}(\lambda) \psi = (\partial_\nu u)_{|\Gamma_N},
\end{equation}
where $u$ is the unique solution of (\ref{z1}). On the other hand, since $v = c^{n-2} u$ is the unique solution of (\ref{z4}), we have
$$
  \Lambda_{g, V_{g,c,\lambda},\Gamma_D,\Gamma_N}(\lambda) \psi = (\partial_\nu v)_{|\Gamma_N} = \big((\partial_\nu c^{n-2}) u + c^{n-2} \partial_\nu u \big)_{|\Gamma_N}.
$$
Since $c = 1$ and $u = \psi$ on $\Gamma_N$, we thus obtain
\begin{equation} \label{z6}
  \Lambda_{g,V_{g,c,\lambda},\Gamma_D,\Gamma_N}(\lambda) \psi = \big((\partial_\nu c^{n-2}) \psi + \partial_\nu u \big)_{|\Gamma_N}.
\end{equation}
If $\Gamma_D \cap \Gamma_N = \emptyset$, which is Case 1 in our Proposition, we have $\psi = 0$ on $\Gamma_N$. Hence we obtain
\begin{equation} \label{z7}
  \Lambda_{g,V_{g,c,\lambda},\Gamma_D,\Gamma_N}(\lambda) \psi = \big(\partial_\nu u \big)_{|\Gamma_N} = \Lambda_{c^4 g,\Gamma_D,\Gamma_N}(\lambda) \psi.
\end{equation}
If $\Gamma_D \cap \Gamma_N \ne \emptyset$ and $\partial_\nu c = 0$ on $\Gamma_N$, which is Case 2, we also get
\begin{equation} \label{z8}
  \Lambda_{g,V_{g,c,\lambda},\Gamma_D,\Gamma_N}(\lambda) \psi = \big(\partial_\nu u \big)_{|\Gamma_N} = \Lambda_{c^4 g,\Gamma_D,\Gamma_N}(\lambda) \psi.
\end{equation}
\end{proof}

Proposition \ref{Link-c-to-V} gives a clear link between the anisotropic Calder\'on problems \textbf{(Q2)} and \textbf{(Q3)}. As an application and by way of a conclusion for this sub-section, let us show for instance how \textbf{(Q3)} implies \textbf{(Q2)} in the case of local data, \textit{i.e.} $\Gamma_D = \Gamma_N = \Gamma$ any open subset in $\partial M$.

\begin{prop} \label{Q3-to-Q2}
  If $\Gamma_D = \Gamma_N = \Gamma$ is any open set in $\partial M$ and $\lambda \in \R$, then \textbf{(Q3)} implies \textbf{(Q2)}.
\end{prop}

\begin{proof}
Assume that \textbf{(Q3)} holds and assume that for two metrics $g$ and $c^4g$, we have
\begin{equation} \label{u1}
  \Lambda_{c^4 g, \Gamma}(\lambda) = \Lambda_{g, \Gamma}(\lambda),
\end{equation}
where $\Lambda_{c^4 g, \Gamma}(\lambda)$ stands for $\Lambda_{c^4 g, \Gamma, \Gamma}(\lambda)$. Then by local boundary determination (\cite{DSFKSU, KY, LeU}, we can conclude that $c_{|\Gamma} = 1$ and $\left( \partial_{\nu} c \right)_{|\Gamma} = 0$. Hence, we can use (\ref{Link}) to show that (\ref{u1}) is equivalent to (with the previously defined notations)
 \begin{equation} \label{u2}
  \Lambda_{g, V_{g,c,\lambda}, \Gamma}(\lambda) = \Lambda_{g, 0, \Gamma}(\lambda),
\end{equation}
with $V_{g,c,\lambda}$ given by (\ref{Vgc}). Finally, our hypothesis that \textbf{(Q3)} holds true now implies that $V_{g,c,\lambda} = 0$, or in other words that
$$
  \Delta_g c^{n-2} + \lambda (1-c^4) c^{n-2} = 0.
$$	
Since $c^{n-2}_{|\Gamma} = 1$, $\left( \partial_{\nu} c^{n-2} \right)_{|\Gamma} = 0$ and $c$ is bounded, unique continuation principle for 2nd order elliptic PDE on a smooth manifold with smooth boundary (see \cite{Ho}, Section 28 or \cite{Ta}, Theorem 4) shows that $c = 1$ on $M$ and \textbf{(Q2)} is proved.
\end{proof}

\subsection{A brief survey of known results on the anisotropic Calder\'on problem}

The most comprehensive results known on the anisotropic Calder\'on problems \textbf{(Q1)}, \textbf{(Q2)} and \textbf{(Q3)} pertain to the case of \emph{zero frequency}, that is $\lambda = 0$, under the hypotheses of full data ($\Gamma_D = \Gamma_N = \partial M$) or local data ($\Gamma_D = \Gamma_N = \Gamma$ with $\Gamma$ any open subset of $M$). In dimension $2$, the anisotropic Calder\'on problem \textbf{(Q1)} for global and local data with $\lambda = 0$ has been given a positive answer for compact connected Riemannian surfaces in \cite{LaU, LeU}. We also refer to \cite{ALP} for similar results answering \textbf{(Q1)} for global and local data in the case of anisotropic conductivities which are only $L^\infty$ on bounded domains of $\R^n$.

A positive answer to \textbf{(Q1)} for global and local data and zero frequency $\lambda = 0$ in dimension $3$ or higher has been given for compact connected real analytic Riemannian manifolds with real analytic boundary, satisfying certain topological assumptions, in \cite{LeU}. These assumptions were later weakened in \cite{LaU, LaTU}. Similarly, \textbf{(Q1)} has been answered positively for compact connected Einstein manifolds with boundary in \cite{GSB}.

The general anisotropic Calder\'on problem \textbf{(Q1)} in dimension $n\geq 3$  full or local data is still a major open problem. Some important results on the special cases covered by questions \textbf{(Q2)} and \textbf{(Q3)} have been obtained recently in \cite{DSFKSU, DSFKLS, KS1} for classes of smooth compact connected Riemannian manifolds with boundary that are called \emph{admissible}. Such manifolds $(M,g)$ are \emph{conformally transversally anisotropic}, meaning that
$$
  M \subset \subset \R \times M_0, \quad g = c ( e \oplus g_0),
$$
where $(M_0,g_0)$ is a $n-1$ dimensional smooth compact connected Riemannian manifold with boundary, $e$ is the Euclidean metric on the real line and $c$ is a smooth strictly positive function in the cylinder $\R \times M_0$. Furthermore the transverse manifold $(M_0, g_0)$ is assumed to be \emph{simple}\footnote{A compact manifold $(M_0,g_0)$ is said to be simple if any two points in $M_0$ can be connected by a unique geodesic depending smoothly on the endpoints, and if $\partial M_0$ is strictly convex as a submanifold of $(M,g) = c ( e \oplus g_0)$, meaning that its second fundamental form is positive definite.}. It has been shown in \cite{DSFKSU, DSFKLS} that for admissible manifolds, the conformal factor $c$ is uniquely determined from the knowledge of the DN map at zero frequency $\lambda = 0$, so that both \textbf{(Q2)} and \textbf{(Q3)} have positive answers in this context. These results have been further extended to the case of partial data in \cite{KS1} (see below). We also refer to \cite{GT1, Is, IUY1} for additional results in the  case of local data and to the surveys \cite{GT2, KS2} for further references.

There are also positive results  for problem \textbf{(Q3)} in the case of bounded domains $\Omega$ of $\R^n, \ n \geq 3$ equipped with the Euclidean metric, for data measured on distinct subsets $\Gamma_D, \Gamma_N$ of $\partial M$ which are not assumed to be disjoint, \cite{KSU}. The requirement here is that the sets $\Gamma_D, \Gamma_N$ where the measurements are made must overlap, in the sense that $\Gamma_D \subset \partial \Omega$ can possibly have very small measure, in which case $\Gamma_N$ must have slightly larger measure than $\partial \Omega \setminus \Gamma_D$. These results have been generalized in \cite{KS1} to the case of admissible Riemannian manifolds, where use is made of the fact that admissible manifolds admit \emph{limiting Carleman weights}\footnote{We refer to \cite{DSFKSU} for the definition and properties of limiting Carleman weights on manifolds and their applications.} $\varphi$. Thanks to the existence of $\varphi$, we can decompose the boundary of $M$ as
$$
  \partial M = \partial M_+ \cup \partial M_{\textrm{tan}} \cup \partial M_-,
$$
where
$$
  \partial M_\pm = \{ x \in \partial M: \ \pm \partial_\nu \varphi(x) > 0 \}, \quad \partial M_{\textrm{tan}} = \{ x \in \partial M: \  \partial_\nu \varphi(x) = 0 \}.
$$
In essence, the authors of \cite{KS1} show that the answer to \textbf{(Q3)} is positive\footnote{In fact, additional geometric assumptions on the transverse manifold $(M_0,g_0)$ are needed to give a full proof of this result. We refer to \cite{KS1} Theorem 2.1 for the precise statement.} if the set of Dirichlet data $\Gamma_D$ contains $\partial M_- \cup \Gamma_a$ and the set of Neumann measurements $\Gamma_N$ contains $\partial M_+ \cup \Gamma_a$ where $\Gamma_a$ is some open subset of $\partial M_{\textrm{tan}}$. Hence in particular, the sets $\Gamma_D$ and $\Gamma_N$ must overlap in order to have uniqueness. The only exception occurs in the case where $\partial M_{\textrm{tan}}$ has zero measure, in which case it is enough to take $\Gamma_D = \partial M_-$ and $\Gamma_N = \partial M_+$ to have uniqueness in \textbf{(Q3)} (see Theorem 2.3 of \cite{KS1}). Note in this case that $\Gamma_D \cap \Gamma_N = \partial M_- \cap \partial M_+ = \emptyset$.

Only a few results are known in the case  of data measured on \emph{disjoint sets}, and these apply to the case of zero frequency $\lambda = 0$. Besides the paper \cite{KS1} which concerns a certain subclass of admissible Riemannian manifolds, the only other result we are aware is due to Imanuvilov, Uhlmann and Yamamoto \cite{IUY2} which applies to the $2$-dimensional case, and concerns the potential of a Schr\"odinger equation on a two-dimensional domain homeomorphic to a disc. It is shown that when the boundary is partitioned into eight clockwise-ordered arcs $\Gamma_1, \Gamma_2, \dots, \Gamma_8$, then the potential is determined by boundary measurements with sources supported on $S = \Gamma_2 \cup \Gamma_6$ and fields observed on $R = \Gamma_4 \cup \Gamma_8$, hence answering \textbf{(Q3)} positively in this special setting.

Finally, we mention some related papers by Rakesh \cite{Rak}, by Oksanen, Lassas \cite{LO1, LO2} and by Kurylev, Oksanen, Lassas \cite{KLO} , which are concerned with the \emph{hyperbolic} anisotropic Calder\'on problem, which amounts to the case in which the partial DN map is assumed to be known at all frequencies $\lambda$. We refer to \cite{KKL} for a detailed discussion of the hyperbolic anisotropic Calder\'on problem and to \cite{KKLM} for the link between the hyperbolic DN map and the elliptic DN map at all frequencies. We also mention the work of Rakesh \cite{Rak}, who proved that the coefficients of a wave equation on a one-dimensional interval are determined by boundary measurements with sources supported on one end of the interval and the waves observed on the other end. Here again, the uniqueness result entails to know the hyperbolic DN map or equivalently the DN map at all frequencies.


\subsection{Main results}

In our previous paper \cite{DKN2}, we showed that the answers to \textbf{(Q2)} (and thus \textbf{(Q1)}) as well as \textbf{(Q3)} were negative when the Dirichlet and Neumann data are measured on disjoint sets of the boundary. Within the class of \emph{rotationally invariant toric cylinders} of dimensions $2$ and $3$, we constructed an infinite number of pairs of non isometric metrics and potentials having the same partial DN maps when $\Gamma_D \cap \Gamma_N = \emptyset$ and for any fixed frequency $\lambda$ not belonging to the Dirichlet spectra of the corresponding Laplace-Beltrami or Schr\"odinger operators. With respect to the inverse problems \textbf{(Q1)} and \textbf{(Q2)}, an interesting fact was that any pair of such metrics turned out to belong to the same conformal class, where the corresponding conformal factor had to satisfy a certain nonlinear ODE.

In Section \ref{1}, we explain the hidden mechanism behind the results of \cite{DKN2} and as a consequence, construct counterexamples to uniqueness for the anisotropic Calder\'on problem for any smooth compact connected Riemannian manifold with boundary, of dimension higher than $3$, with Dirichlet data and Neumann data measured on disjoint subsets $\Gamma_D$ and $\Gamma_N$ such that $\overline{\Gamma_D \cup \Gamma_N} \ne \partial M$. More precisely, we highlight a subtle gauge invariance admitted by the anisotropic Calder\'on problem with disjoint sets satisfying the above assumption. This gauge invariance is given by certain conformal rescalings of a fixed metric $g$ by a conformal factor that satisfies a nonlinear elliptic PDE of Yamabe type with appropriate boundary conditions (see Theorem \ref{Main-1}). We are able to find smooth positive solutions of this nonlinear equation of Yamabe type using the standard technique of lower and upper solutions. We emphasize that this technique works thanks to the crucial assumption $\overline{\Gamma_D \cup \Gamma_N} \ne \partial M$, that allows us to play on the boundary conditions appearing in the nonlinear equation. The main results of Section \ref{1} are Theorem \ref{Main-1} and Definition \ref{Gauge0}.

In Section \ref{2}, we pursue our analysis by considering the anisotropic Calder\'on problem \textbf{(Q3)} with disjoint sets. We first show that the gauge invariance for the anisotropic Calder\'on problem \textbf{(Q2)} turns out not to be a gauge invariance for the problem \textbf{(Q3)} through the link established in Proposition \ref{Link-c-to-V}. In fact, given a fixed potential $V = V_{g,c,\lambda}$ as in (\ref{Vgc}), there exist infinitely many conformal factors $\tilde{c}$ such that $V_{g,\tilde{c},\lambda} = V$. We show that this family of conformal factors $\tilde{c}$ precisely corresponds to the whole gauge associated to the metric $c^4 g$ in the sense of Definition \ref{Gauge0}. Second, recall that despite of the lack of gauge invariance for the problem \textbf{(Q3)}, non trivial counterexamples to uniqueness for the problem \textbf{(Q3)} were found in \cite{DKN2} within the class of rotationally invariant toric cylinders. In the core of Section \ref{2}, we improve our previous construction and find a large class of new counterexamples to uniqueness for the problem \textbf{(Q3)}. This class consists in cylindrical Riemannian manifolds having two ends, \textit{i.e.} whose boundary consists in two disconnected components, and equipped with a warped product metric. We show non-uniqueness for \textbf{(Q3)} when the Dirichlet and Neumann data belong to distinct connected components of the boundary, a requirement which turns out to be crucial. This is done in Theorem \ref{NonUniquenessQ3}.

In Section \ref{3}, we come back to the anisotropic Calder\'on problem \textbf{(Q2)} and use the counterexamples to uniqueness for the problem \textbf{(Q3)} found in Section \ref{2} to construct counterexamples to uniqueness for the problem \textbf{(Q2)} which do not arise from the gauge invariance defined in Section \ref{1}. To do this, we make crucial use of the link between \textbf{(Q2)} and \textbf{(Q3)} stated in Proposition \ref{Link-c-to-V}. The main point here is to construct from a fixed frequency $\lambda$ and a fixed potential $V$ satisfying certain conditions a conformal factor $c$ such that $V = V_{g,c,\lambda}$ as in (\ref{Vgc}). This amounts to solving a nonlinear elliptic equation of Yamabe type of the same type as the one considered in Section \ref{1}. This is done once again using the lower and upper solutions technique. We stress the fact that the counterexamples to uniqueness for the problem \textbf{(Q2)} obtained in this way are still cylindrical Riemannian manifolds having two ends and that the Dirichlet and Neumann data are measured on distinct connected components of the boundary. The main result in this Section is Theorem \ref{NonUniquenessQ4}.

Finally, in Section \ref{4}, we summarize our results and conjecture some additional results concerning the anisotropic Calder\'on problem with disjoint sets depending on the connectedness or not of the boundary.


\Section{The gauge invariance for the anisotropic Calder\'on problem in dimension $n\geq 3$} \label{1}

Throughout this Section, we assume that $\dim M \geq 3$. The result of the following proposition relies on the simple observation that there is a subtle gauge invariance behind the anisotropic Calder\'on problem when the Dirichlet and Neumann data are measured on disjoint sets. This gauge invariance is given by certain conformal rescalings of a fixed metric $g$ by a strictly positive smooth function that satisfies a nonlinear elliptic PDE of Yamabe type (see (\ref{Main-EDP})).

\begin{prop} \label{Main}
Let $(M,g)$ be a smooth compact connected Riemannian manifold of dimension $n\geq 3$ with smooth boundary $\partial M$ and let $\lambda \in \R$ not belong to the Dirichlet spectrum $\sigma(-\Delta_g)$. Let $\Gamma_D, \Gamma_N$ be open sets of $\partial M$ such that $\Gamma_D \cap \Gamma_N = \emptyset$. If there exists a smooth strictly positive function $c$ satisfying
\begin{equation} \label{Main-EDP}
  \left\{ \begin{array}{cc} \Delta_{g} c^{n-2} + \lambda ( c^{n-2} - c^{n+2}) = 0, & \textrm{on} \ M, \\
	c = 1, & \textrm{on} \ \Gamma_D \cup \Gamma_N, \end{array} \right.
\end{equation}
then the conformally rescaled Riemannian metric $\tilde{g} = c^4 g$ satisfies
$$
  \Lambda_{\tilde{g},\Gamma_D, \Gamma_N}(\lambda) = \Lambda_{g,\Gamma_D, \Gamma_N}(\lambda).  	
$$
\end{prop}

\begin{proof}
Consider the Dirichlet problem at fixed frequency $\lambda$ associated to $\tilde{g} = c^4 g$, \textit{i.e.}
\begin{equation} \label{a1}
  \left\{ \begin{array}{cc} -\Delta_{\tilde{g}} u = \lambda u, & \textrm{on} \ M, \\ u = \psi, & \textrm{on} \ \Gamma_D, \\ u = 0, & \textrm{on} \ \partial M \setminus \Gamma_D. \end{array} \right.
\end{equation}
As in the proof of Proposition \ref{Link-c-to-V} and thanks to our assumptions on $\Gamma_D$ and $\Gamma_N$, it is immediate to see that the function $v = c^{n-2} u$ satisfies
\begin{equation} \label{a3}
  \left\{ \begin{array}{cc} (-\Delta_{g} + V_{g,c,\lambda}) v = \lambda v, & \textrm{on} \ M, \\ v = c^{n-2} \psi, & \textrm{on} \ \Gamma_D, \\ v = 0, & \textrm{on} \ \partial M \setminus \Gamma_D, \end{array} \right.
\end{equation}
where $V_{g,c,_\lambda}$ is given by (\ref{Vgc}). Assume now that there exists a smooth positive function $c: M \longrightarrow \R^{+*}$ satisfying
\begin{equation} \label{Cond-c}
  \left\{ \begin{array}{rcl} V_{g,c,\lambda} & = & 0, \ \textrm{on} \ M, \\
  c & = & 1 \ \textrm{on} \ \Gamma_D \cup \Gamma_N. \end{array} \right.
\end{equation}
Using (\ref{Vgc}), these conditions can be written as the nonlinear Dirichlet problem for $w = c^{n-2}$
\begin{equation} \label{EDPc}
  \left\{ \begin{array}{cc} \Delta_{g} w + \lambda (w - w^{\frac{n+2}{n-2}}) = 0, & \textrm{on} \ M, \\ w = \eta, & \textrm{on} \ \partial M, \end{array} \right.
\end{equation}
where $\eta = 1$ on $\Gamma_D \cup \Gamma_N$. Note that (\ref{EDPc}) is nothing but the PDE (\ref{Main-EDP}) in the statement of the Proposition.

Assuming the existence of a positive solution $w$ of (\ref{EDPc}) and thus of the corresponding conformal factor $c = w^{\frac{1}{n-2}}$ of (\ref{Cond-c}), the function $v = c^{n-2} u$ satisfies
\begin{equation} \label{a4}
  \left\{ \begin{array}{cc} -\Delta_{g} v = \lambda v, & \textrm{on} \ M, \\ v = \psi, & \textrm{on} \ \Gamma_D, \\ v = 0, & \textrm{on} \ \partial M \setminus \Gamma_D. \end{array} \right.
\end{equation}
Therefore, the function $v$ is the unique solution of the Dirichlet problem (\ref{a1}) at fixed frequency $\lambda$ for the metric $g$. We conclude that
$$
\Lambda_{\tilde{g},\Gamma_D, \Gamma_N}(\lambda) = \Lambda_{g,\Gamma_D, \Gamma_N}(\lambda),
$$
as in the proof of Proposition \ref{Link-c-to-V}.
\end{proof}

\begin{rem}
  Using the well-known fact that the potential $q_{g,c}$ in (\ref{q}) can be expressed as
	\begin{equation} \label{ScalarCurvature}
	  q_{g,c} = \frac{n-2}{4(n-1)} \left( Scal_g - c^4 \, Scal_{c^4 g} \right),
	\end{equation}
	where $Scal_g$ and $Scal_{c^4 g}$ denote the scalar curvatures associated to $g$ and $\tilde{g} = c^4 g$ respectively, the nonlinear PDE (\ref{Main-EDP}) satisfied by the conformal factor $c$ may be re-expressed in more geometric terms by observing that $c$ will satisfy (\ref{Main-EDP}) is and only if
\begin{equation} \label{GeometricInterpretation}
  Scal_{c^4 g} = \frac{Scal_g + \frac{4(n-1)}{n-2} \lambda (1-c^4)}{c^4}.
\end{equation}
\end{rem}

In view of Proposition \ref{Main}, we see that in order to construct counterexamples to uniqueness for the anisotropic Calder\'on problem on a smooth compact Riemannian manifold $(M,g)$ of dimension $n\geq 3$ with smooth boundary $\partial M$, where the  Dirichlet and Neumann data measured on disjoint subsets of the boundary, it is sufficient to find a conformal factor $c$ satisfying the nonlinear PDE of Yamabe type (\ref{Main-EDP}) and such that $c \ne 1$ on $M$ (see \ref{Inv-Conformal-1}). We shall see below that this can been done by using the well-known technique of lower and upper solutions.

Indeed, recall that we are interested in solutions $w = c^{n-2}$ of the nonlinear elliptic PDE (see (\ref{EDPc})):
\begin{equation} \label{Eqw}
  \left\{ \begin{array}{cc} \Delta_g w + f(w) =0 , & \textrm{on} \ M, \\ w = \eta, & \textrm{on} \ \partial M, \end{array} \right.
\end{equation}
where $f(w) = \lambda (w-w^{\frac{n+2}{n-2}})$ and $\eta$ is a smooth function on $\partial M$ such that $\eta = 1$ on $\Gamma_D \cup \Gamma_N$. We may thus more generally consider the nonlinear Dirichlet problem
\begin{equation} \label{GeneralDP}
  \left\{ \begin{array}{cc} \Delta_g w + f(x,w) =0 , & \textrm{on} \ M, \\ w = \eta, & \textrm{on} \ \partial M, \end{array} \right.
\end{equation}
where $f$ is a smooth function on $M \times \R$ and $\eta$ is a smooth function on $\partial M$. We recall the definitions of an upper solution and a lower solution of (\ref{GeneralDP}).

\begin{defi}
An upper solution ${\overline{w}}$ is a function in $C^2(M) \cap C^0(\overline{M})$ satisfying
\begin{equation}\label{upper}
\Delta_g {\overline{w}}+ f(x,{\overline{w}}) \leq 0 \  \textrm{on} \ M, \quad \textrm{and} \quad {\overline{w}}_{|\partial M} \geq \eta.
\end{equation}
Similarly, a lower solution ${\underline{w}}$ is a function in $C^2(M) \cap C^0(\overline{M})$ satisfying
\begin{equation}\label{under}
\Delta_g {\underline{w}}+ f(x,{\underline{w}}) \geq 0  \ \textrm{on} \ M, \quad \textrm{and} \quad {\underline{w}}_{|\partial M} \leq \eta.
\end{equation}
\end{defi}

It is well-known (see \cite{Sat}, Thm 2.3.1. or \cite{Ta2}, Section 14.1) that if we can find a lower solution ${\underline{w}}$ and an upper solution ${\overline{w}}$ satisfying ${\underline{w}} \leq {\overline{w}}$ on $M$, then there exists a solution $w \in C^{\infty}(\overline{M})$ of (\ref{GeneralDP}) such that ${\underline{w}} \leq w \leq {\overline{w}}$ on $M$. For completeness, let us briefly sketch the construction of such a solution : we pick $\mu>0$ such that $|\partial_w f(x,w)| \leq \mu$ for $w \in [\min \  {\underline{w}} , \max \ {\overline{w}}]$.  Then, we define recursively a sequence $(w_k)$ by $w_0 = {\underline{w}}$, $w_{k+1} = \Phi(w_k)$ where $\Phi(w) = \varphi$ is given by solving
\begin{equation}
\Delta_g \varphi - \mu \varphi = -\mu w - f(x,w) \ ,\ \varphi_{|\partial M} = \eta.
\end{equation}
Using the maximum principle, we see that this sequence satisfies
\begin{equation}\label{sequence}
{\underline{w}}=w_0 \leq w_1 \leq \cdots \leq w_k \cdots \leq {\overline{w}}.
\end{equation}
We deduce that $w = \displaystyle\lim_{k \to \infty} w_k$ is a solution of (\ref{Eqw}). The details of the construction are given in the above references \cite{Sat, Ta2}.

Now, we can establish the following elementary result.

\begin{prop} \label{NonlinearDirichletPb}

For all $\lambda \geq 0$, (resp. for all $\lambda < 0$), and for all smooth positive functions $\eta$ such that $\eta \ne 1$ on $\partial M$, (resp. $\eta \lneq 1$ on $\partial M$),  there exists a positive solution $w \in C^{\infty}(\overline{M})$ of (\ref{Eqw}) satisfying $w \ne 1$ on $M$.

\end{prop}

\begin{proof} 1. Assume first that $\lambda \geq 0$. \\
a) If $\eta \gneq 1$, then  ${\underline{w}}=1$ is a lower solution and ${\overline{w}}= \max \eta$ is an upper solution of (\ref{Eqw}). Moreover, they clearly satisfy  ${\underline{w}}  \leq {\overline{w}}$. \\
b) Likewise, if $0 < \eta \lneq 1$, then ${\underline{w}}= \min \eta$ is a lower solution and ${\overline{w}}= 1$ is an upper solution of (\ref{Eqw}). They still satisfy ${\underline{w}}  \leq {\overline{w}}$. \\
c) Finally, if $0 < \min \eta < 1 < \max \eta$, then ${\underline{w}}= \min \eta$ is a lower solution and ${\overline{w}}= \max \eta$ is an upper solution of (\ref{Eqw}). Moreover, they satisfy  ${\underline{w}}  \leq {\overline{w}}$. \\
2. Assume now that $\lambda < 0$ and $0 < \eta \lneq 1$. \\
We define ${\underline{w}}$ as the unique solution of the Dirichlet problem
\begin{equation} \label{Dir1}
  \left\{ \begin{array}{cc} \Delta_g \underline{w} + \lambda \underline{w} =0 , & \textrm{on} \ M, \\ \underline{w} = \eta, & \textrm{on} \ \partial M. \end{array} \right.
\end{equation}
Since $\lambda < 0$, the strong maximum principle implies that $0 <  \underline{w} \leq \max \eta$ on $M$, (see \cite{GT}, Corollary 3.2 and Theorem 3.5). Moreover,
$$\triangle_g \underline{w} + \lambda (\underline{w} - (\underline{w})^{\frac{n+2}{n-2}}) = -\lambda (\underline{w})^{\frac{n+2}{n-2}} \geq 0.$$
It follows that  $\underline{w}$ is a lower solution of (\ref{Eqw}).
We then define ${\overline{w}}$ as the unique solution of the Dirichlet problem
\begin{equation} \label{Dir2}
  \left\{ \begin{array}{cc} \Delta_g \overline{w} + \lambda \overline{w} = \lambda (\max \eta)^{\frac{n+2}{n-2}} , & \textrm{on} \ M, \\ \overline{w} = \eta, & \textrm{on} \ \partial M. \end{array} \right.
\end{equation}
According to the maximum principle, we also have $0 \leq \overline{w}$ on $M$. Now, setting $v= \overline{w}- \max \eta$, and since $\eta \leq 1$, we have
\begin{equation} \label{Dir11}
  \left\{ \begin{array}{cc} \Delta_g   v + \lambda v = \lambda (\max \eta^{\frac{n+2}{n-2}}- \max \eta)  \geq 0 , & \textrm{on} \ M, \\ v = \eta - \max \eta \leq 0, & \textrm{on} \ \partial M. \end{array} \right.
\end{equation}
So, according to the maximum principle again, we deduce that $v \leq 0$ on $M$, or equivalently $\overline{w} \leq \max \eta$. We deduce as previously that $\overline{w}$ is an upper solution of (\ref{Eqw}). Finally, $\overline{w} - \underline{w}$ satisfies
\begin{equation} \label{Dir3}
  \left\{ \begin{array}{cc} \Delta_g (\overline{w} - \underline{w}) + \lambda (\overline{w} - \underline{w}) = \lambda (\max \eta)^{\frac{n+2}{n-2}} < 0 , & \textrm{on} \ M, \\ \overline{w} - \underline{w} = 0, & \textrm{on} \ \partial M. \end{array} \right.
\end{equation}
Then, the maximum principle implies again $\overline{w} \geq \underline{w}$ which finishes the proof.
\end{proof}

In order to use the existence results of Proposition \ref{NonlinearDirichletPb} for the construction of a conformal factor $c$ satisfying (\ref{Main-EDP}) and $c \ne 1$ on $M$, we need to be able to choose $\eta \ne 1$ on $\partial M$. We thus make the crucial assumption on the disjoint Dirichlet and Neumann data that
\begin{equation} \label{Main-Hyp}
  \overline{\Gamma_D \cup \Gamma_N} \ne \partial M.
\end{equation}
Putting together then the results of Proposition \ref{Main} and Proposition \ref{NonlinearDirichletPb}, we have proved

\begin{thm} \label{Main-1}
  Let $(M,g)$ be a smooth compact connected Riemannian manifold of dimension $n\geq 3$ with smooth boundary $\partial M$. Let $\Gamma_D, \Gamma_N$ be open subsets of $\partial M$ such that $\Gamma_D \cap \Gamma_N = \emptyset$ and $\overline{\Gamma_D \cup \Gamma_N} \ne \partial M$. Consider a conformal factor $c \ne 1$ on $M$ whose existence is given in Proposition \ref{NonlinearDirichletPb}, defined as a smooth solution of the nonlinear Dirichlet problem
\begin{equation} \label{Main-EDP-1}
  \left\{ \begin{array}{cc} \Delta_{g} c^{n-2} + \lambda (c^{n-2} - c^{n+2}) = 0, & \textrm{on} \ M, \\
	c^{n-2} = \eta, & \textrm{on} \ \partial M, \end{array} \right.
\end{equation}
where $\eta$ is a suitable smooth positive function on $\partial M$ satisfying $\eta = 1$ on $\Gamma_D \cup \Gamma_N$ and $\eta \ne 1$ on $\partial M \setminus (\Gamma_D \cup \Gamma_N)$. Then the Riemannian metric $\tilde{g} = c^4 g$ with $c \ne 1$ on $M$ satisfies
$$
  \Lambda_{\tilde{g},\Gamma_D, \Gamma_N}(\lambda) = \Lambda_{g,\Gamma_D, \Gamma_N}(\lambda).  	
$$
\end{thm}

This gauge invariance for the anisotropic Calder\'on problem with disjoint data can be formalized in the following way.

\begin{defi}[Gauge invariance] \label{Gauge0}
  Let $(M,g)$ and $(M,\tilde{g})$ be smooth compact connected Riemannian manifolds of dimension $n\geq 3$ with smooth boundary $\partial M$. Let $\lambda \in \R$  not belong to the union of the Dirichlet spectra of $-\Delta_g$ and $-\Delta_{\tilde{g}}$. Let $\Gamma_D, \Gamma_N$ be open subsets of $\partial M$ such that $\Gamma_D \cap \Gamma_N = \emptyset$ and $\overline{\Gamma_D \cup \Gamma_N} \ne \partial M$. We say that $g$ and $\tilde{g}$ are gauge related if there exists a smooth positive conformal factor $c$ such that: \\
\begin{equation} \label{Gauge}
   \left\{ \begin{array}{rl} \tilde{g} & = c^4 g, \\
	                           \Delta_{g} c^{n-2} + \lambda (c^{n-2} - c^{n+2}) & = 0, \textrm{on} \ M, \\
	                           c & = 1, \textrm{on} \ \Gamma_D \cup \Gamma_N, \\
														 c & \ne 1, \textrm{on} \ \partial M \setminus (\Gamma_D \cup \Gamma_N).
														\end{array} \right.
   \end{equation}
In that case, we have: $\Lambda_{\tilde{g}, \Gamma_D, \Gamma_N}(\lambda) = \Lambda_{g, \Gamma_D, \Gamma_N}(\lambda)$.
\end{defi}

\begin{rem}
  In dimension $2$, the gauge invariance described in Definition \ref{Gauge0} for the anisotropic Calder\'on problem with disjoint data is not relevant except for the case of zero frequency. Indeed, the nonlinear PDE (\ref{Gauge}) that should satisfy the conformal factor $c$ becomes
\begin{equation} \label{EDP-Dim2}
  \lambda (1 - c^4) = 0, \ \textrm{on} \ M.
\end{equation}
In other words, $c$ must be identically equal to $1$ if $\lambda \ne 0$. Recalling that in dimension $2$ and for zero frequency, a conformal transformation is already known to be a gauge invariance of the anisotropic Calder\'on problem, we see that our construction will not lead to new counterexamples to uniqueness in dimension $2$, for any frequency $\lambda$.
\end{rem}

We conclude this Section by stating a version of the anisotropic Calder\'on conjecture with disjoint data modulo the previously defined gauge invariance. \\

\noindent \textbf{(Q4)} \emph{Let $M$ be a smooth compact connected manifold with smooth boundary $\partial M$ and let $g,\, \tilde{g}$ be smooth Riemannian metrics on $M$. Let $\Gamma_D, \Gamma_N$ be any open sets of $\partial M$ such that $\Gamma_D \cap \Gamma_N = \emptyset$ and $\lambda \in \R$ not belong to $\sigma(-\Delta_g) \cup \sigma(-\Delta_{\tilde{g}})$. If $\Lambda_{g,\Gamma_D, \Gamma_N}(\lambda) = \Lambda_{\tilde{g},\Gamma_D, \Gamma_N}(\lambda)$, is it true that $g = \tilde{g}$ up to the gauge invariances: \\
1. (\ref{Inv-Diff}) in any dimension, \\
2. (\ref{Inv-Conf}) if $\dim M = 2$ and $\lambda = 0$}, \\
3. (\ref{Gauge}) if $\dim M \geq 3$ and $\overline{\Gamma_D \cup \Gamma_N} \ne \partial M$?


\Section{The anisotropic Calder\'on problem for Schr\"odinger operators in dimension $n\geq 2$} \label{2}

In this Section, we consider the anisotropic Calder\'on problem \textbf{(Q3)} for Schr\"odinger operators on a fixed smooth compact connected Riemannian manifold $(M,g)$ of dimension $n\geq 2$, with smooth boundary $\partial M$, under the assumption that the Dirichlet and Neumann data are measured on disjoint subsets of the boundary. We first show that the previously constructed counterexamples to uniqueness for the anisotropic Calder\'on problem \textbf{(Q2)} in dimension higher than $3$ cannot be used to construct counterexamples to uniqueness for \textbf{(Q3)} through the link (\ref{Link}). To this effect, we start by proving the following elementary lemma:

\begin{lemma}\label{lemmafactor}
Let $(M,g)$ be a smooth compact connected Riemannian manifolds of dimension $n\geq 3$ with smooth boundary $\partial M$. Consider two smooth conformal factors $c_1$ and $c_2$ such that $c:= \frac{c_2}{c_1}$ satisfies
\begin{equation}\label{factor}
\Delta_{c_1^4 g} c^{n-2} + \lambda (c^{n-2} - c^{n+2}) = 0 \ \rm{on}\  M.
\end{equation}
Then,
\begin{equation} \label{c2}
  V_{g,c_1,\lambda} = V_{g,c_2,\lambda}.
\end{equation}
\end{lemma}

\begin{proof}
Using (\ref{ConformalScaling}) and (\ref{q}) with the conformal factor $c_1$, we obtain easily
 \begin{equation}
 (\Delta_g -q_{g,c_1}) c_2^{n-2} + \lambda \left( c_1^4 c_2^{n-2}-c_2^{n+2} \right) =0.
 \end{equation}
 So, using (\ref{q}) again with the conformal factor $c_2$, we get
 \begin{equation}
 (q_{g, c_2} - q_{g,c_1}) + \lambda \left( c_1^4 -c_2^4\right) =0,
 \end{equation}
 or equivalently
$$
  V_{g,c_1,\lambda} = V_{g,c_2,\lambda}.
$$
\end{proof}

\vspace{0.2cm}

 As a consequence, let $\Gamma_D, \Gamma_N$ be open subsets of $\partial M$ such that $\Gamma_D \cap \Gamma_N = \emptyset$. Consider two smooth conformal factors $c_1$ and $c_2$ such that the metrics $G=c_1^4 g $ and $\tilde{G}=c_2^4 g$ are gauge equivalent in the sense of Definition \ref{Gauge0}, \textit{i.e.} $\Lambda_{G, \Gamma_D, \Gamma_N}(\lambda) = \Lambda_{\tilde{G}, \Gamma_D, \Gamma_N}(\lambda)$. Then, we obtain from (\ref{Link}) that
$$
  \Lambda_{g, V_{g,c_1,_\lambda}, \Gamma_D, \Gamma_N}(\lambda) = \Lambda_{g, V_{g,c_2,\lambda}, \Gamma_D, \Gamma_N}(\lambda),
$$
but Lemma \ref{lemmafactor} implies that  $ V_{g,c_1,\lambda} = V_{g,c_2,\lambda}$. Thus, the gauge invariance for the anisotropic Calder\'on problem \textbf{(Q2)} with disjoint data highlighted in Section \ref{1} is not a gauge invariance for the corresponding anisotropic Calder\'on problem \textbf{(Q3)}. In other words, we just showed that the gauge invariance for \textbf{(Q2)} corresponds in fact to all the possible conformal factors $c$ satisfying $V_{g,c,\lambda} = V_{g,c_0,\lambda} = q$ for a fixed conformal factor $c_0$ or related potential $q$.

\vspace{0.5cm}

Nevertheless, we exhibited in \cite{DKN2} some constructive counterexamples to uniqueness for the anisotropic Calder\'on problem \textbf{(Q3)} with disjoint sets on smooth compact connected Riemannian toric cylinders equipped with a warped product metric in dimensions $2$ or $3$. Precisely, we refer to Theorems 3.2 and 3.4 in \cite{DKN2} for counterexamples of \textbf{(Q2)} and \textbf{(Q3)} respectively in dimension $2$ and to Theorem 4.7 in \cite{DKN2} for counterexamples of \textbf{(Q3)} in dimension $3$.

In this Section, we generalize the results of \cite{DKN2} and show that the same type of constructive counterexamples to uniqueness can be obtained for any smooth compact connected Riemannian cylinder $M$ having two ends (meaning that the boundary $\partial M$ consists in two connected components), equipped with a warped product metric. More precisely, we consider the general model in which $M = [0,1]\times K$, where $K$ is an arbitrary $(n-1)$-dimensional closed manifold, equipped with a Riemannian metric of the form
\begin{equation} \label{Metric}
  g = f^4(x) [dx^2 + g_K],
\end{equation}
where $f$ is a smooth strictly positive function on $[0,1]$ and $g_K$ denotes a smooth Riemannian metric on $K$. Clearly, $(M,g)$ is a $n$-dimensional warped product cylinder and the boundary $\partial M$ has two connected components, namely $\partial M = \Gamma_0 \cup \Gamma_1$ where $\Gamma_0 = \{0\} \times K$ and $\Gamma_1 = \{1\} \times K$ correspond to the two ends of $(M,g)$ . The positive Laplace-Beltrami operator on $(M,g)$ has the expression
\begin{equation} \label{Laplacian}
  -\Delta_g = f^{-(n+2)} \left( -\partial_x^2 - \triangle_K + q_f(x) \right) f^{n-2},
\end{equation}
where $-\triangle_K$ denotes the positive Laplace-Beltrami operator on $(K,g_K)$ and $q_f = \frac{(f^{n-2})''}{f^{n-2}}$.

Let us consider a potential $V = V(x) \in L^\infty(M)$ and $\lambda \in \R$ such that $\lambda \notin \{ \lambda_j\}_{j \geq 1}$ where $\{ \lambda_j\}_{j \geq 1}$ is the Dirichlet spectrum of $-\Delta_g + V$. We are interested in the unique solution $u$ of the Dirichlet problem
\begin{equation} \label{eq0}
  \left\{ \begin{array}{rcl}
	(-\triangle_g + V) u & = & \lambda u, \ \textrm{on} \ M, \\
	u & = & \psi, \ \textrm{on} \ \partial M.
	\end{array} \right.
\end{equation}
Thanks to (\ref{Laplacian}) and setting $v = f^{n-2} u$, this can be written as
\begin{equation} \label{Eq1}
  \left\{ \begin{array}{rcl}
	\left[ -\partial^2_x - \triangle_K + q_f + (V-\lambda) f^4 \right] v & = & 0, \ \textrm{on} \ M, \\
	v & = & f^{n-2} \psi, \ \textrm{on} \ \partial M.
	\end{array} \right.
\end{equation}

In order to construct the DN map corresponding to the problem (\ref{eq0}), we shall use the following notations. Since the boundary $\partial M$ of $M$ has two disjoint components $\partial M = \Gamma_0 \cup \Gamma_1$, we can decompose the Sobolev spaces $H^s(\partial M)$ as $H^s(\partial M) = H^s(\Gamma_0) \oplus H^s(\Gamma_1)$ for any $s \in \R$ and we shall use the vector notation
$$
 \varphi = \left( \begin{array}{c} \varphi^0 \\ \varphi^1 \end{array} \right),
$$
to denote the elements $\varphi$ of $H^s(\partial M) = H^s(\Gamma_0) \oplus H^s(\Gamma_1)$. The DN map is a linear operator from $H^{1/2}(\partial M)$ to $H^{-1/2}(\partial M)$ and thus has the structure of an operator valued $2 \times 2$ matrix
$$
  \Lambda_g(\lambda) = \left( \begin{array}{cc} \Lambda_{g,\Gamma_0,\Gamma_0}(\lambda) & \Lambda_{g,\Gamma_1, \Gamma_0}(\lambda) \\ \Lambda_{g,\Gamma_0, \Gamma_1}(\lambda) & \Lambda_{g,\Gamma_1,\Gamma_1}(\lambda) \end{array} \right),
$$
whose components are operators from $H^{1/2}(K)$ to $H^{-1/2}(K)$.

Now we use the warped product structure of $(M,g)$ and the fact that $V = V(x)$ to find a simple expression of the DN map by decomposing all the relevant quantities onto a Hilbert basis of harmonics $(Y_k)_{k \geq 0}$ of the Laplace-Beltrami operator $-\triangle_K$ on the closed manifold $K$. We first write $\psi = (\psi^0, \psi^1) \in H^{1/2}(\Gamma_0) \times H^{1/2}(\Gamma_1)$ using their Fourier expansions as
$$
 \psi^0 = \sum_{k \geq 0} \psi^0_k Y_k, \quad \psi^1 = \sum_{k \geq 0} \psi^1_k Y_k.
$$
Note that for any $s \in \R$, the space $H^{s}(K)$ can be described as
$$
  \varphi \in H^{s}(K) \ \Longleftrightarrow \ \left\{ \varphi \in \D'(K), \ \varphi = \sum_{k \geq 0} \varphi_k Y_k, \quad \sum_{k \geq 0} (1 + \mu_k)^{s} |\varphi_k|^2 < \infty \ \right\},
$$
where $0 = \mu_0 < \mu_1 \leq \mu_2 \leq \dots$ are the eigenvalues of $-\triangle_K$.

Now the unique solution $v$ of (\ref{Eq1}) takes the form
$$
  v = \sum_{k \geq 0} v_k(x) Y_k(\omega),
$$
where the functions $v_k$ are the unique solutions of the boundary value problems given by
\begin{equation} \label{Eq2}
  \left\{ \begin{array}{c} -v_k'' + [ q_f + (V-\lambda) f^4] v_k = -\mu_k v_k, \ \textrm{on} \ [0,1], \\
   v_k(0) = f^{n-2}(0) \psi^0_k, \quad v_k(1) = f^{n-2}(1) \psi^1_k. \end{array} \right.
\end{equation}
Moreover the DN map can be diagonalized in the Hilbert basis $\{ Y_k \}_{k \geq 0}$ and thus shown to take the following convenient expression
\begin{equation} \label{DN1}
  \Lambda_{g,V}(\lambda)_{|<Y_k>} = \Lambda^k_{g,V}(\lambda)
	= \left( \begin{array}{c} \frac{(n-2) f'(0)}{f^{n+1}(0)} v_k(0) - \frac{v_k'(0)}{f^n(0)} \\
		-\frac{(n-2) f'(1)}{f^{n+1}(1)} v_k(1) + \frac{v_k'(1)}{f^n(1)} \end{array} \right) .
\end{equation}

Let us now interpret the quantities $v_k(0), v_k'(0), v_k(1), v_k'(1)$ in terms of the boundary values of $v_k$. For this, we introduce the characteristic and Weyl-Titchmarsh functions of the boundary value problem
\begin{equation} \label{Eq3}
  \left\{ \begin{array}{c} -v'' + [q_{f}(x) +(V-\lambda) f^4(x)] v = - \mu v, \\
  v(0) = 0, \quad v(1) = 0.  \end{array}\right.
\end{equation}
Note that the equation (\ref{Eq3}) is nothing but equation (\ref{Eq2}) in which the angular momentum $-\mu_k$ is written as $-\mu$ and is interpreted as the \emph{spectral parameter} of the equation. Since the potential $q_f +(V-\lambda) f^4 \in L^1([0,1])$ and is real, we can define for all $\mu \in \C$ two fundamental systems of solutions of (\ref{Eq3})
$$
  \{ c_0(x,\mu), s_0(x,\mu)\}, \quad \{ c_1(x,\mu), s_1(x,\mu)\},
$$
by imposing the Cauchy conditions
\begin{equation} \label{FSS}
  \left\{ \begin{array}{cccc} c_0(0,\mu) = 1, & c_0'(0,\mu) = 0, & s_0(0,\mu) = 0, & s_0'(0,\mu) = 1, \\
 	                  c_1(1,\mu) = 1, & c'_1(1,\mu) = 0, & s_1(1,\mu) = 0, & s'_1(1,\mu) = 1. \end{array} \right.
\end{equation}

\begin{rem} \label{Utile1}
In terms of the Wronskian $W(u,v) = uv' - u'v$, we have
$$
  W(c_0,s_0) = 1, \quad W(c_1,s_1) = 1
$$
Moreover, we remark (see \cite{PT}) that the functions $\mu \to c_j(x,\mu), \, s_j(x,\mu)$ and their derivatives with respect to $x$ are entire functions of order $\frac{1}{2}$.
\end{rem}

Following \cite{DKN2}, we define the characteristic function of (\ref{Eq3}) by
\begin{equation} \label{Char}
  \Delta_{g,V}(\mu) = W(s_0, s_1),
\end{equation}
and the Weyl-Titchmarsh functions by
\begin{equation} \label{WT}
  M_{g,V}(\mu) = - \frac{W(c_0, s_1)}{\Delta_{g,V}(\mu)}, \quad N_{g,V}(\mu) = - \frac{W(c_1, s_0)}{\Delta_{g,V}(\mu)}.
\end{equation}

\begin{rem} \label{Utile2}
 1. Since the function $\Delta_{g,V}$ is entire, its zeros form a discrete set in $\C$. We denote this set by $(\alpha_k)_{k \geq 1}$ and remark that they correspond to "minus"\footnote{since the spectral parameter of (\ref{Eq3}) is $-\mu$.} the Dirichlet spectrum of the 1D Schr\"odinger operator $-\frac{d^2}{dx^2} + [q_f + (V-\lambda) f^4]$. Moreover, these zeros are simple, (see Theorem 2, p. 30 of \cite{PT}).\\
2. The functions $M_{g,V}$ and $N_{g,V}$ are meromorphic with poles given by $(\alpha_k)_{k \geq 1}$. Under our assumption that $\lambda$ does not belong to the Dirichlet spectrum of $-\triangle_g + V$, we can show that the eigenvalues $(\mu_k)_{k \geq 1}$ of $-\triangle_K$ cannot be poles of $M_{g,V}$ and $N_{g,V}$. In particular, $0$ is not a pole of $M_{g,V}$ and $N_{g,V}$. We refer to \cite{DKN2}, Remark 3.1, for the detailed proof of this assertion.
\end{rem}

Writing the solution $v_k$ of (\ref{Eq3}) as
$$
  v_k(x) = \alpha \,c_0(x,\mu_k) + \beta \,s_0(x,\mu_k) = \gamma \,c_1(x,\mu_k) + \delta \,s_1(x,\mu_k),
$$
for some constants $\alpha,\beta,\gamma,\delta$, a straightforward calculation as in \cite{DKN2}, section 4, shows that the DN map $\Lambda_g^k(\lambda)$ on each harmonic $Y_k, \, k \geq 0$ has the expression
\begin{equation} \label{DN-Partiel}
  \Lambda^k_g(\lambda) = \left( \begin{array}{cc} \frac{(n-2)f'(0)}{f^3(0)} - \frac{M_{g,V}(\mu_k)}{f^2(0)} & -\frac{f^{n-2}(1)}{f^n(0) \Delta_{g,V}(\mu_k)} \\ -\frac{f^{n-2}(0)}{f^n(1) \Delta_{g,V}(\mu_k)} & -\frac{(n-2)f'(1)}{f^3(1)} - \frac{N_{g,V}(\mu_k)}{f^2(1)} \end{array} \right) .
\end{equation}
Hence the DN map $\Lambda_{g,V}^k(\lambda)$ on each harmonic $Y_k$ is simply a multiplication by a $2 \times 2$ matrix whose coefficients are expressed in terms of some boundary values of the metric $g$ and its first normal derivative $\partial_\nu g$ as well as the characteristic function $\Delta_{g,V}$ for the anti-diagonal components and the Weyl-Titchmarsh functions $M_{g,V} $and $N_{g,V}$ for the diagonal components evaluated at the $\{\mu_k\}_{k \geq 0}$ which are the eigenvalues of the Laplacian $-\triangle_K$ on the closed manifold $K$. Note that the non locality of the DN map is seen through the multiplication by the functions $\Delta_{g,V}(\mu_k), M_{g,V}(\mu_k)$ and $N_{g,V}(\mu_k)$ since they depend on the whole potential $q_f + (V-\lambda) f^4$ and thus on the whole metric $g$ and potential $V$.

Let us now come back to the study of the anisotropic Calder\'on problem \textbf{(Q3)} for metrics (\ref{Metric}) and potentials $V = V(x)$ when the Dirichlet and Neumann data are measured on disjoint sets of the boundary. Assume precisely that $\Gamma_D, \Gamma_N$ are open subsets of $\partial M$ that belong to distinct connected components of $\partial M$. For instance, if we choose $\Gamma_D \subset \Gamma_0$ and $\Gamma_N \subset \Gamma_1$, then the measured partial DN map $\Lambda_{g,V,\Gamma_D, \Gamma_N}(\lambda)$ is given by
\begin{equation} \label{DN2}
  \Lambda_{g,V,\Gamma_D, \Gamma_N}(\lambda) \psi = - \left( \sum_{k \geq 0} \frac{f^{n-2}(0)}{f^n(1) \Delta_{g,V}(\mu_k)} \psi_k Y_k \right)_{|\Gamma_N},
\end{equation}
where $\psi = \sum_{k \geq 0} \psi_k Y_k$ and supp$\,\psi \subset \Gamma_D$. It is clear from the expression (\ref{DN2}) that the characteristic function $\Delta_{g,V}$ is the essential quantity that determines uniquely $\Lambda_{g,V,\Gamma_D, \Gamma_N}(\lambda)$ when $\Gamma_D$ and $\Gamma_N$ belong to distinct connected components of the boundary. We thus consider the following question: can we find potentials $\tilde{V}$ distinct from $V$ and such that $\Delta_{g,V}(\mu) = \Delta_{g,\tilde{V}}(\mu)$ for all $\mu \in \C$? In the positive case, we will thus have found counterexamples to uniqueness for the Calder\'on problem \textbf{(Q3)} with disjoint data.

The answer is yes and is provided by the following key Lemma.

\begin{lemma} \label{Link-Iso}
  Let $g$ be a fixed metric as in (\ref{Metric}) and $V = V(x), \tilde{V} = \tilde{V}(x) \in L^\infty(M)$. Then
$$
	\Delta_{g,V}(\mu) = \Delta_{g,\tilde{V}}(\mu), \quad \forall \mu \in \C,
$$
if and only if
$$
  q_f + (V-\lambda)f^4 \ \textrm{and} \ q_f + (\tilde{V}-\lambda)f^4 \ \textrm{are isospectral for} \ (\ref{Eq3}).
$$	
\end{lemma}

\begin{proof}
  We recall first from Remark \ref{Utile1} that the FSS $(c_j(x,\mu), s_j(x,\mu)), \ j=1,2$ are entire of order $\frac{1}{2}$ with respect to $\mu$. Hence we deduce easily from (\ref{Char}) that $\Delta_{g,V}, \Delta_{g,\tilde{V}}$ are also entire of order $\frac{1}{2}$. Moreover, we know from Remark \ref{Utile2} that $0$ is not a zero of $\Delta_{g,V}$ and $\Delta_{g,\tilde{V}}$. It follows then from the Hadamard factorization Theorem (see for instance \cite{Lev}) that
\begin{equation} \label{r1}
  \Delta_{g,V}(\mu) = C \prod_{k \geq 1} \left( 1 - \frac{\mu}{\alpha_k} \right), \quad \Delta_{g,\tilde{V}}(\mu) = \tilde{C} \prod_{k \geq 1} \left( 1 - \frac{\mu}{\tilde{\alpha_k}} \right),
\end{equation}
where $(\alpha_k)_{k \geq 1}, \ (\tilde{\alpha_k})_{k \geq 1}$ denote "minus" the Dirichlet spectra of the 1D Schr\"odinger operators $-\frac{d^2}{dx^2} + [q_f + (V-\lambda) f^4]$ and $-\frac{d^2}{dx^2} + [q_f + (\tilde{V}-\lambda) f^4]$ respectively (see Remark \ref{Utile1} again) and $C, \tilde{C}$ are constants.

Second, it turns out that $\Delta_{g,V}$ and $\Delta_{g,\tilde{V}}$ have universal asymptotics when $\mu \to \infty$. Precisely, we know from \cite{PT} and \cite{DKN2}, Corollary 2.1 that
\begin{equation} \label{r2}
  \Delta_{g,V}(\mu), \, \Delta_{g,\tilde{V}}(\mu) \sim \frac{\sinh(\sqrt{\mu})}{\sqrt{\mu}}, \quad \mu \to \infty.
\end{equation}

As a consequence, we deduce from (\ref{r1}) that if $\Delta_{g,V}(\mu) = \Delta_{g,\tilde{V}}(\mu)$ for all $\mu \in \C$, then $\alpha_k = \tilde{\alpha_k}$ for all $k \geq 1$. This means precisely that the potentials $q_f + (V-\lambda)f^4$ and $q_f + (\tilde{V}-\lambda)f^4$ are isospectral for the boundary value problem (\ref{Eq3}). Conversely, if we assume that $q_f + (V-\lambda)f^4$ and $q_f + (\tilde{V}-\lambda)f^4$ are isospectral for (\ref{Eq3}), then $\alpha_k = \tilde{\alpha_k}$ for all $k \geq 1$. This means using (\ref{r1}) that $\Delta_{g,V}(\mu) = \frac{C}{\tilde{C}} \Delta_{g,\tilde{V}}(\mu)$ for all $\mu \in \C$. But the universal asymptotics (\ref{r2}) imply then that $C = \tilde{C}$. Hence $\Delta_{g,V} = \Delta_{g,\tilde{V}}$.
\end{proof}

Thanks to the fundamental results of P\"oschel and Trubowitz \cite{PT}, Theorem 5.2, we have a complete description of the class of isospectral potentials for the Schr\"odinger operator with Dirichlet boundary conditions (\ref{Eq3}). This result shows that for each eigenfunction $\phi_k, \ k \geq 1$ of (\ref{Eq3}), we can find a one parameter family of explicit potentials  isospectral to $Q(x) = q_f + (V-\lambda) f^4 \in L^2([0,1])$ by the formula
\begin{equation} \label{Iso1}
  Q_{k,t}(x) = Q(x) - 2 \frac{d^2}{dx^2} \log \theta_{k,t}(x), \quad \quad \forall t \in \R,
\end{equation} 	
where
\begin{equation} \label{Iso2}
  \theta_{k,t}(x) = 1 + (e^t - 1) \int_x^1 \phi_k^2(s) ds.
\end{equation} 		
Using the definition $Q(x) = q_f + (V-\lambda) f^4$, we get the explicit one parameter families of potentials $\tilde{V}$
\begin{equation} \label{IsoPot}
  \tilde{V}_{k,t}(x) = V(x) - \frac{2}{f^4(x)} \frac{d^2}{dx^2} \log \theta_{k,t}(x), \quad \forall k \geq 1, \quad \forall t \in \R,
\end{equation}
where $\theta_{k,t}$ is given by (\ref{Iso2}). Using Lemma \ref{Link-Iso} and (\ref{DN2}), we have proved

\begin{thm} \label{NonUniquenessQ3}
Let $(M,g)$ be a cylindrical warped product as in (\ref{Metric}), $V=V(x) \in L^\infty(M)$ and $\lambda \in \R$ not belong to the Dirichlet spectrum of $-\triangle_g + V$. Then the family of potentials $\tilde{V}_{k,t}$ defined in (\ref{IsoPot}) for all $k \geq 1$ and $t \in \R$ satisfies
$$
  \Lambda_{g,V,\Gamma_D,\Gamma_N}(\lambda) = \Lambda_{g,\tilde{V}_{k,t},\Gamma_D,\Gamma_N}(\lambda),
$$
whenever $\Gamma_D$ and $\Gamma_N$ are open sets that belong to different connected components of $\partial M$.
\end{thm}
We emphasize that the non-uniqueness result of the Theorem holds when $\Gamma_D = \Gamma_0$ and $\Gamma_N = \Gamma_1$, hence when $\overline{\Gamma_D \cup \Gamma_N} = \partial M$.

\begin{rem} \label{Rem-Iso}
\begin{itemize}
\item The potentials $\tilde{V}_{k,t}$ have the same regularity properties as $V$ on $[0,1]$ for all $k \geq 1$ and for all $t \in \R$. Indeed, the normalized eigenfunctions $\phi_k(x)$ are smooth on $[0,1]$ by elliptic regularity. Hence, the functions $\theta_{k,t}$ are also smooth and never vanish on $[0,1]$ for all $k \geq 1$ and for all $t \in \R$ by (\ref{Iso2}). In particular, if $V$ is smooth on $[0,1]$, then $\tilde{V}_{k,t}$ is also smooth by (\ref{IsoPot}).

\item For all $k \geq 1$ and for all $t \in \R$, $\tilde{V}_{k,t}(0) = V(0)$ and $\tilde{V}_{k,t}(1) = V(1)$. This follows from a short calculation using (\ref{Iso2}) and (\ref{IsoPot}).

\item If moreover $V > 0$ (resp. $V<0$), then for all $k \geq 1$, there exists $T_k > 0$ such that $\tilde{V}_{k,t} >0$ (resp. $\tilde{V}_{k,t} < 0$) for all $-T_k < t < T_k$. Indeed, it is clear that for a fixed $k \geq 1$, the function $2 \frac{d^2}{dx^2} \log \theta_{k,t}(x)$ can be made arbitrarily small as $t \to 0$ uniformly w.r.t. $x \in [0,1]$. The result follows thanks to (\ref{IsoPot}).

\end{itemize}
\end{rem}

\begin{rem} \label{WT-vs-Char}
The preceding construction fails when $\Gamma_D, \Gamma_N$ belong to the same connected component of the boundary $\partial M$. This is due to the fact that on each harmonic $Y_k$, the associated partial DN map $\Lambda_{g,\Gamma_D,\Gamma_N}(\lambda)$ acts essentially as an operator of multiplication by the Weyl-Titchmarsh functions $M_{g,V}(\mu_k)$ or $N_{g,V}(\mu_k)$ (see (\ref{DN2})) instead of the characteristic function $\Delta_{g,V}(\mu_k)$. But as it is well known in 1D inverse spectral theory, the Weyl-Titchmarsh functions contain much more information than the characteristic function. This is the object of the Borg-Marchenko Theorem (see \cite{Be, Bo1, Bo2, ET, FY, GS, KST}). In particular, for rotationally invariant toric cylinders of dimensions 2 and 3, we showed in (\cite{DKN2}, Theorems 3.4 and 4.6), that if $\Gamma_D$ and $\Gamma_N$ belong to the same connected component of the boundary $\partial M$ {\footnote{with a technical assumption on the size of $\Gamma_N$}}, then $\Lambda_{g,V,\Gamma_D,\Gamma_N}(\lambda) = \Lambda_{g,\tilde{V},\Gamma_D,\Gamma_N}(\lambda)$ implies $V = \tilde{V}$.
\end{rem}

\Section{Counterexamples to uniqueness for the anisotropic Calder\'on problem with disjoint data in dimension $n\geq 3$, modulo the gauge invariance} \label{3}

In this Section, we show that the counterexamples to uniqueness given in Theorem \ref{NonUniquenessQ3} for the anisotropic Calder\'on problem \textbf{(Q3)} lead to non trivial counterexamples to uniqueness for the anisotropic Calder\'on problem \textbf{(Q2)} in dimension $n\geq 3$ modulo the gauge invariance introduced in Section \ref{1}, Definition \ref{Gauge0}. To do this, we have in mind Proposition \ref{Link-c-to-V} which gives a clear link between the anisotropic Calder\'on problems \textbf{(Q2)} and \textbf{(Q3)} when $\Gamma_D \cap \Gamma_N = \emptyset$.

More precisely, we fix $(M,g)$ a cylindrical warped product as in (\ref{Metric}), $V=V(x) \in C^\infty(M)$ and $\lambda \in \R$ not belonging to the Dirichlet spectrum of $-\triangle_g + V$. Given a potential $\tilde{V}$ given by (\ref{IsoPot}), we would like to try to construct conformal factors $c$ and $\tilde{c}$ in such a way that (see (\ref{Vgc}) for the notations)
$$
  V_{g,c,\lambda} = V, \quad V_{g,\tilde{c},\lambda} = \tilde{V},
$$
and
$$
  c, \tilde{c} = 1 \ \textrm{on} \ \Gamma_D \cup \Gamma_N.
$$
If we manage to construct such conformal factors $c$ and $\tilde{c}$, then Theorem \ref{NonUniquenessQ3} and Proposition \ref{Link-c-to-V} would imply immediately that
$$
  \Lambda_{c^4 g, \Gamma_D, \Gamma_N}(\lambda) = \Lambda_{\tilde{c}^4 g, \Gamma_D, \Gamma_N}(\lambda)
$$
whenever $\Gamma_D \cap \Gamma_N = \emptyset$. Moreover, the metrics $c^4 g$ and $\tilde{c}^4 g$ wouldn't be gauge related in the sense of Definition \ref{Gauge0} since they are associated to different potentials $V \ne \tilde{V}$ (see Lemma \ref{lemmafactor} and the paragraph just after).

Considering only the problem of finding $c > 0$ satisfying $V_{g,c,\lambda} = V$, $c = 1$ on $\Gamma_D \cup \Gamma_N$, we see from (\ref{Vgc}) that is sufficient to find a smooth positive solution $w$ of the nonlinear Dirichlet problem
\begin{equation} \label{DirichletPb-q}
  \left\{ \begin{array}{rl} \triangle_g w + (\lambda - V)w - \lambda w^{\frac{n+2}{n-2}} & = 0, \ \textrm{on} \ M, \\
	                          w & = \eta, \ \textrm{on} \ \partial M,
	\end{array} \right. 														
\end{equation}
where $\eta = 1$ on $\Gamma_D \cup \Gamma_N$ and $0\eta >0$ on $\partial M$.

For zero frequency $\lambda = 0$, the nonlinear Dirichlet problem (\ref{DirichletPb-q}) becomes linear, so that the usual existence and uniqueness Theorem for a Dirichlet problem on a Riemannian manifold with boundary as well as the strong maximum principle can be used to prove

\begin{prop}[Zero frequency] \label{q-to-c-0}
  Assume that $\lambda = 0$ and $V \geq 0$ on $M$. Then for each positive smooth function $\eta$ on $\partial M$ such that $\eta = 1$ on $\Gamma_D \cup \Gamma_N$, there exists a unique smooth positive solution $w$ of (\ref{DirichletPb-q}) such that $0 <  w \leq \max \eta$ on $M$.
\end{prop}

We now turn to the case of  frequency $\lambda \in \R$, and prove the following:

\begin{prop}[general case] \label{q-to-c-1}
  1. If $\lambda >0$  and  $0 <V(x) <\lambda$ on $M$, then for each positive function $\eta$ on $\partial M$ such that $\max \eta \geq 1$ on $\partial M$, there exists a smooth positive solution $w$ of (\ref{DirichletPb-q}). \\
	2. If $\lambda \leq 0$ and $V(x) \geq 0$ on $M$, then for each for each positive function $\eta$ on $\partial M$ such that $\eta \leq 1$ on $\partial M$, there exists a smooth positive solution $w$ of (\ref{DirichletPb-q}). \\
	
\end{prop}

\begin{proof}
1. We use again the technique of lower and upper solutions. We define ${\underline{w}}= \epsilon$ where $\epsilon>0$ is small enough. We have
\begin{equation}
\Delta_g \underline{w} + (\lambda-V) \underline{w}  -\lambda (\underline{w})^{\frac{n+2}{n-2}} =
\epsilon \left( (\lambda-V) -\lambda \epsilon^{{\frac{n+2}{n-2}} -1} \right) >0,
\end{equation}
so $\underline{w}$ is a lower solution. In the same way, we define $\overline{w} = \max \eta$ and we have
\begin{equation}
\Delta_g \overline{w} + (\lambda-V) \overline{w}  -\lambda (\overline{w})^{\frac{n+2}{n-2}} =
\lambda ( \max \eta - \max \eta^{\frac{n+2}{n-2}} ) - V \max \eta \leq 0.
\end{equation}
It follows that $\overline{w}$ is an upper solution and clearly $\underline{w} \leq \overline{w}$.

\vspace{0.2cm}
2. In the case $\lambda \leq 0$, $V \geq 0$ and $\eta \leq 1$, we define ${\underline{w}}$ as the unique solution of the Dirichlet problem
\begin{equation} \label{Dir4}
  \left\{ \begin{array}{cc} \Delta_g \underline{w} + (\lambda-V) \underline{w} = 0 , & \textrm{on} \ M, \\ \underline{w} = \eta, & \textrm{on} \ \partial M. \end{array} \right.
\end{equation}
Since $(\lambda - V) \leq 0$, the strong maximum principle implies that $0 < \underline{w} \leq \max \eta$ on $M$. Moreover, $\triangle_g \underline{w} + (\lambda - V) \underline{w} - \lambda (\underline{w})^{\frac{n+2}{n-2}} = -\lambda (\underline{w})^{\frac{n+2}{n-2}} \geq 0$. Hence $\underline{w}$ is a lower solution of (\ref{DirichletPb-q}).

\vspace{0.2cm} \noindent
Now, we define ${\overline{w}}$ as the unique solution of the Dirichlet problem
\begin{equation} \label{Dir5}
  \left\{ \begin{array}{cc} \Delta_g \overline{w} + (\lambda - V) \overline{w} = (\lambda-V) (\max \eta)^{\frac{n+2}{n-2}} , & \textrm{on} \ M, \\ \overline{w} = \eta, & \textrm{on} \ \partial M. \end{array} \right.
\end{equation}
According to the maximum principle, we also have $\overline{w} \geq 0$ on $M$. Setting $v = \overline{w}- \max \eta$, we see that
\begin{equation}
\Delta_g v + (\lambda-V) v = (\lambda-V)(\max \eta^{\frac{n+2}{n-2}} - \max \eta)  \geq 0,
\end{equation}
since $\eta \leq 1$. Hence, the maximum principle implies that $v \leq 0$ on $M$, or equivalently $ \overline{w} \leq \max \eta$.

\vspace{0.2cm}

We deduce that
\begin{equation}
\Delta_g \overline{w} + (\lambda - V) \overline{w} - \lambda \overline{w}^{\frac{n+2}{n-2}} =
(\lambda -V) (\max \eta^{\frac{n+2}{n-2}} - \overline{w}^{\frac{n+2}{n-2}}) - V \overline{w}^{\frac{n+2}{n-2}} \leq 0,
\end{equation}
since $V$ is positive. Thus, $\overline{w}$ is an upper solution of (\ref{DirichletPb-q}).

Finally, $\overline{w} - \underline{w}$ satisfies
\begin{equation} \label{Dir6}
  \left\{ \begin{array}{cc} \Delta_g (\overline{w} - \underline{w}) + (\lambda - V) (\overline{w} - \underline{w}) = (\lambda-V) (\max \eta)^{\frac{n+2}{n-2}} < 0 , & \textrm{on} \ M, \\ \overline{w} - \underline{w} = 0, & \textrm{on} \ \partial M. \end{array} \right.
\end{equation}
Then, the maximum principle implies again $\overline{w} \geq \underline{w}$. Then according to the lower and upper solutions technique, there exists a smooth positive solution $w$ of (\ref{DirichletPb-q}) such that $w \ne 1$ on $M$. \\

\end{proof}

\vspace{0.2cm}
Let us now come back to the geometric setting of Theorem \ref{NonUniquenessQ3}. Here $M = [0,1] \times K$ is equipped with a warped product metric $g$ as in  (\ref{Metric}). First, let us fix a frequency $\lambda \in \R$. 

\vspace{0.2cm}
1. Assume that $\lambda >0$. Consider a potential $V = V(x) \in C^\infty(M)$ such that $0<V(x)<\lambda$ and such that $\lambda$ does not belong to the Dirichlet spectrum of $-\Delta_g +V$. This is always possible since the discrete spectrum of $-\Delta_g + V$ is unstable under small perturbations of $V$. Now, consider a potential $\tilde{V} = \tilde{V}_{k,t}(x)$ as in (\ref{IsoPot}) and such that $0<\tilde{V}(x) < \lambda$. Observe that this can always been achieved for small enough $-\epsilon < t < \epsilon$ thanks to the definition (\ref{IsoPot}) of $\tilde{V}_{k,t}$ (see Remark \ref{Rem-Iso}). At last, consider a smooth positive function $\eta$ on $\partial M$ such that $\eta =1$ on $\Gamma_D \cup \Gamma_N$ and such that $\max \eta \geq 1$. Then, Proposition \ref{q-to-c-1} implies the existence of smooth positive conformal factors $c$ and $\tilde{c}$ such that
$$
  V_{g,c,\lambda} = V, \quad c = 1 \ \textrm{on} \ \Gamma_D \cup \Gamma_N,
$$
and
$$
  V_{g,\tilde{c},\lambda} = \tilde{V}, \quad \tilde{c} = 1 \ \textrm{on} \ \Gamma_D \cup \Gamma_N.
$$
But from Theorem \ref{NonUniquenessQ3}, we have
$$
  \Lambda_{g, V, \Gamma_D, \Gamma_N}(\lambda) = \Lambda_{g, \tilde{V}, \Gamma_D, \Gamma_N}(\lambda).
$$
Therefore from Proposition \ref{Link-c-to-V}, we conclude that
$$
  \Lambda_{c^4 g, \Gamma_D, \Gamma_N}(\lambda) = \Lambda_{\tilde{c}^4 g, \Gamma_D, \Gamma_N}(\lambda).
$$

\vspace{0.2cm}
2. Assume that $\lambda \leq 0$. Consider a potential $V(x)>0$ and a smooth positive function $\eta$ on $\partial M$ such that $\eta =1$ on $\Gamma_D \cup \Gamma_N$ and such that $\eta \leq 1$. Clearly, $\lambda$ does not belong to the Dirichlet spectrum of $-\Delta_g +V$. Then, we follow the same stategy as in the previous case.

\vspace{0.2cm}

We emphasize that the metrics $c^4 g$ and $\tilde{c}^4 g$ aren't connected by the invariance gauge of Section \ref{1} since they correspond to different potentials $V = V_{g,c,\lambda}$ and $\tilde{V} = V_{g,\tilde{c},\lambda}$. Hence we have constructed a large class of counterexamples to uniqueness for the anisotropic Calder\'on problem when the Dirichlet and Neumann data are measured on disjoint sets of the boundary \emph{modulo this gauge invariance}.

Therefore we have proved:

\begin{thm} \label{NonUniquenessQ4}
  Let $M = [0,1] \times K$ be a cylindrical manifold having two ends equipped with a warped product metric $g$ as in (\ref{Metric}). Let $\Gamma_D, \Gamma_N$ be open sets that belong to different connected components of $\partial M$. Let $\lambda \in \R$ be a fixed frequency.  Then there exists an infinite number of smooth positive conformal factors $c$ and $\tilde{c}$ on $M$ with aren't gauge equivalent in the sense of Definition \ref{Gauge0} such that
$$
  \Lambda_{c^4 g, \Gamma_D, \Gamma_N}(\lambda) = \Lambda_{\tilde{c}^4 g, \Gamma_D, \Gamma_N}(\lambda).
$$	
\end{thm}

\Section{Conclusions and open problems} \label{4}

In this paper, we have highlighted a natural gauge invariance for the anisotropic Calder\'on problem on smooth compact connected Riemannian manifolds, that arises in the case of disjoint data. We refer to Definition \ref{Gauge0} for the definition of the gauge invariance that led to the formulation \textbf{(Q4)} of the anisotropic Calder\'on conjecture. Moreover, we managed to construct some explicit counterexamples to uniqueness for \textbf{(Q4)}, \textit{i.e.} modulo this gauge invariance, within the class $(M,g)$ of cylindrical manifolds with two ends equipped with a warped product metric. This was done in Theorem \ref{NonUniquenessQ3} for Schr\"odinger operators in dimensions $\geq 2$ and in Theorem \ref{NonUniquenessQ4} for the usual anisotropic Calder\'on problem in dimensions $\geq 3$.

The latter counterexamples to uniqueness rely crucially on the fact that the boundary of $(M,g)$ has more than one connected component and that the Dirichlet and Neumann data are measured on distinct connected components of the boundary. This can easily be seen from the expression (\ref{DN2}) of the associated DN map. On the one hand, the expression of the partial DN map when $\Gamma_D, \Gamma_N$ belong to the same connected component of $\partial M$ depends essentially on the Weyl-Titchmarsh operator (\ref{WT}). On the other hand, the expression of the partial DN map when $\Gamma_D, \Gamma_N$ do not belong to the same connected component of $\partial M$ depends essentially on the characteristic operator (\ref{Char}). The latter contains much less information than the former (this fact is encoded in the Borg-Marchenko theorem, see \cite{Be, Bo1, Bo2, ET, FY, GS, KST}) and allows us to construct the above mentioned counterexamples when $\Gamma_D$ and $\Gamma_N$ belong to different connected components of $\partial M$. Finally, we stress the fact that if $\Gamma_D$ and $\Gamma_N$ were disjoint but belonged to the same connected component of $\partial M$, then we would have uniqueness for the the anisotropic Calder\'on problem \textbf{(Q3)} and thus also for \textbf{(Q2)} modulo the gauge invariance (see Remark \ref{WT-vs-Char}). Therefore, we see that the connectedness or non-connectedness of the boundary $\partial M$ plays a critical role in the anisotropic Calder\'on problem with disjoint data. More precisely, we conjecture: \\

\noindent \textbf{(Q5)}: \emph{Let $M$ be a smooth compact connected manifold with smooth boundary $\partial M$ and let $g,\, \tilde{g}$ be smooth Riemannian metrics on $M$. Let $\Gamma_D, \Gamma_N$ be any open sets of $\partial M$ such that $\Gamma_D \cap \Gamma_N = \emptyset$ and suppose that $\lambda \in \R$ does not belong to $\sigma(-\Delta_g) \cup \sigma(-\Delta_{\tilde{g}})$. \\
1. If $\partial M$ is connected and $\Lambda_{g,\Gamma_D, \Gamma_N}(\lambda) = \Lambda_{\tilde{g},\Gamma_D, \Gamma_N}(\lambda)$, then $g = \tilde{g}$ up to the gauge invariances:
 \begin{itemize}
 \item (\ref{Inv-Diff}) in any dimension,
 \item (\ref{Inv-Conf}) if $\dim M = 2$ and $\lambda = 0$,
 \item (\ref{Gauge}) if $\dim M \geq 3$ and $\overline{\Gamma_D \cup \Gamma_N} \ne \partial M$.
\end{itemize}
2. If $\partial M$ is not connected, then there exist metrics $g$ and $\tilde{g}$ not related by one of the above gauge invariances for which $\Lambda_{g,\Gamma_D, \Gamma_N}(\lambda) = \Lambda_{\tilde{g},\Gamma_D, \Gamma_N}(\lambda)$, at least when $\Gamma_D$ and $\Gamma_N$ belong to distinct connected components of the boundary.} \\

\vspace{0.8cm}
\noindent \textbf{Acknowledgements}: The authors would like to warmly thank Yves Dermenjian for suggesting the crucial role of the transformation law of the Laplacian under conformal scaling in the gauge invariance for the Calderon problem with disjoint data and also Gilles Carron for his help in solving the nonlinear PDE of Yamabe type encountered in Sections \ref{1} and \ref{3}.  \\


\end{document}